\documentclass[a4paper,11pt]{amsart}
\PassOptionsToPackage{no-math}{fontspec}
\usepackage[UTF8,scheme=plain,fontset=fandol]{ctex}
\usepackage{amsmath}
\usepackage{amsfonts}

\usepackage{amscd}
\usepackage{amsthm}
\usepackage{amssymb} 
\usepackage{latexsym}
\usepackage{mathrsfs}
\usepackage{euscript}
\usepackage{array}
\usepackage{enumerate}
\usepackage{dsfont}
\usepackage{color}
\usepackage{wasysym}
\usepackage{graphicx}
\usepackage{placeins}

\usepackage[colorlinks,citecolor=red,pagebackref,hypertexnames=false, breaklinks]{hyperref}

\newcommand{\bel}[1]{\begin{equation}\label{#1}}

\newcommand{\be}{\begin{equation}}

\newcommand{\qe}{\end{equation}}

\newcommand{\dd}{\mathrm{d}}

\newcommand{\Hmm}[1]{\leavevmode{\marginpar{\tiny%
$\hbox to 0mm{\hspace*{-0.5mm}$\leftarrow$\hss}%
\vcenter{\vrule depth 0.1mm height 0.1mm width \the\marginparwidth}%
\hbox to
0mm{\hss$\rightarrow$\hspace*{-0.5mm}}$\\\relax\raggedright #1}}}

\theoremstyle{plain}
\newtheorem{thm}{Theorem}[section]
\newtheorem{theorem}[thm]{Theorem}
\newtheorem{lemma}[thm]{Lemma}
\newtheorem*{unnumberedlemma}{Lemma}

\newtheorem{corollary}[thm]{Corollary}

\newtheorem{proposition}[thm]{Proposition}

\theoremstyle{definition}

\newtheorem*{remark}{Remark}

\numberwithin{equation}{section}

\begin{document}
\raggedbottom 

\date{\today}

\title[Backstepping stabilization with a zero speed]{Backstepping Stabilization of a Hyperbolic System with a Zero Characteristic Speed}


\author[Zhiqiang Wang]{Zhiqiang Wang\textsuperscript{1}}
\thanks{\textsuperscript{1}Zhiqiang Wang: School of Mathematical Sciences and Shanghai Key
Laboratory for Contemporary Applied Mathematics, Fudan University, Shanghai
200433, China.  Email: \texttt{wzq@fudan.edu.cn}}
\author[Wenhao Yang]{Wenhao Yang\textsuperscript{2}}
\thanks{\textsuperscript{2}Wenhao Yang: School of Mathematical Sciences, Fudan University,
Shanghai 200433, China.  Email: \texttt{whyang23@m.fudan.edu.cn}}

   %


\begin{abstract}
This paper investigates the boundary stabilization of a linear hyperbolic system with a zero characteristic speed. We address a case that cannot be stabilized by proportional boundary feedback and develop a backstepping design for the critical setting where the zero-speed subsystem has no intrinsic dissipation. To this end, we construct a Volterra transformation leading to a nonstandard integro-differential target system. Its asymptotic stability in $L^2$ is established by both the Lyapunov--LaSalle method and a spectral semigroup argument. The backstepping kernel equations are explicitly solved, yielding an explicit boundary feedback law that asymptotically stabilizes the original system. Numerical simulations are provided to illustrate the effectiveness of the proposed control design.
\end{abstract}

\maketitle 

\begin{center}
\begin{minipage}{0.8\textwidth}

\end{minipage}
\end{center}

\section{Introduction}
 
Hyperbolic systems arise naturally in a wide range of applications,
including fluid dynamics, traffic flow, and elasticity; see
Chapter~1 of \cite{BC16} for representative examples and the
subsequent chapters for systematic treatments of stability and boundary
stabilization; for controllability, see also \cite{Li94}.
In these systems, propagation along characteristics provides a natural mechanism
through which boundary inputs influence the state in the interior
of the domain.

However, when a characteristic speed vanishes, the corresponding
state component has no spatial transport, and it cannot be
directly influenced by boundary inputs. The presence of zero
characteristic speeds therefore introduces additional difficulties and
challenges in controllability and stabilization; see, e.g.,
\cite{LR10,CGW09}. As pointed out in Problem~15 of \cite{LR10},
the situation becomes even more challenging when internal controls are
unavailable and control has to rely only on boundary inputs.

The controllability and stabilization of hyperbolic systems with zero
characteristic speeds have attracted considerable attention.

For controllability problems, when distributed controls are allowed,
Li and Yu \cite{LY03} established exact controllability by applying
distributed controls to the zero-speed components and boundary controls
to the remaining components. Li and Rao \cite{LR09} showed that, subject
to suitable compatibility conditions, exact controllability can be
achieved using distributed controls acting only on the components
associated with the nonzero characteristic speeds. For a simplified
model, they also obtained exact controllability by means of locally
supported switching controls.

When only boundary controls are available, the situation is more
restrictive. In the presence of zero characteristic speeds, Wang and Yu \cite{WY08} obtained partial exact
controllability, where only the components associated with the nonzero
characteristic speeds are controlled. Yu \cite{Yu18} characterized the exact boundary controllability
of a class of linear hyperbolic systems with an identically zero
characteristic speed in terms of a necessary and sufficient condition
on the initial and final states. For certain nonlinear systems, however,
full exact boundary controllability can still be achieved.
Gugat \cite{Gugat03} treated a critical flow model involving a vanishing
characteristic speed, while Coron, Glass and Wang \cite{CGW09} employed
the return method to establish exact boundary controllability for
quasilinear hyperbolic systems whose vanishing characteristic speed can
be driven away from zero.

Existing stabilization results for hyperbolic systems with zero
characteristic speeds may be grouped here into two lines of research:
direct approaches based on the dissipative or spectral properties of
the closed-loop system, and backstepping-based approaches. Among the direct approaches, Yong \cite{Yong19} established exponential
stability for a class of linear hyperbolic systems with zero
characteristic speeds under a structural stability condition.
More recently, Wang and Yao \cite{WY25} obtained local exponential
stability for a broad class of quasilinear hyperbolic systems with
zero characteristic speeds by exploiting a partial dissipative
structure of the source terms.

The work most directly related to our model is that of Hu and Wang
\cite{HW16}. They considered the coupled system
\[
\begin{aligned}
u_t&=\alpha u_x+\beta v,\\
v_t&=\gamma u,
\end{aligned}
\qquad x\in(0,2\pi),
\]
where $u$ is the propagating component and $v$ is the zero-speed
component. They assumed $\alpha<0$ without loss of generality, so that
$x=0$ is the inflow boundary for $u$. For the proportional boundary feedback
$u(t,0)=k u(t,2\pi)$, they showed that the closed-loop system is
asymptotically stable when $\beta\gamma<0$ and $|k|<1$, whereas no
choice of $k$ achieves asymptotic stability when $\beta\gamma\geq0$. Their main proof combines a detailed spectral
analysis with the Riesz basis method, and they also outlined an alternative proof based on the Lyapunov--LaSalle method. Moreover, the resulting stability is not
exponential because a sequence of closed-loop eigenvalues converges to
zero.

An alternative approach is based on backstepping, which uses a Volterra
transformation to map the original system into a stable target system;
see Krstic and Smyshlyaev \cite{KS08} for a systematic treatment of
the method and \cite{survey26} for a recent survey on PDE backstepping
stabilization. The presence of a
zero-speed state makes this construction particularly difficult because
that state cannot be stabilized directly through characteristic
propagation. Consequently, backstepping results for such systems remain
limited; see also Section~3.1.3 of \cite{survey26}.

De Andrade, Vazquez, Karafyllis and Krstic \cite{AVKK22} first
developed a backstepping design for the following $3\times3$ system:
\[
\begin{aligned}
u_t&=-\lambda_1u_x+c_{12}p+c_{13}v,\\
p_t&=\lambda_2p_x+c_{21}u+c_{23}v,\\
v_t&=c_{31}u+c_{32}p-c_{33}v,
\end{aligned}
\]
where $u$ and $p$ are the two counter-propagating states and $v$ is a
scalar zero-speed state. They applied an invertible Volterra
transformation only to the propagating states, leaving $v$ unchanged.
In the target system, the transport states are decoupled and stable,
while the zero-speed state is input-to-state stable with respect to
them because of its pre-existing damping term $-c_{33}v$.

This design was subsequently extended in \cite{AVKK24} to a
$(1+n+1)\times(1+n+1)$ system containing two counter-propagating
scalar states and an $n$-dimensional zero-speed subsystem. Assuming
that the internal dynamics matrix $\Psi$ of the zero-speed subsystem is
Hurwitz, the authors constructed a full-state boundary feedback that
achieves exponential stability in $L^2$. Thus, both designs require the
zero-speed subsystem itself to possess a suitable intrinsic dissipative
structure.

Building upon the backstepping framework of de Andrade et al., recent
studies have addressed the compensation of arbitrarily long input
delays for hyperbolic systems with zero characteristic speeds, first for
a single delayed boundary input \cite{XSP26} and subsequently for
general multi-input systems with asynchronous delays \cite{SLXY26}.

The present work addresses the case $\beta\gamma>0$ for a
$2\times2$ hyperbolic system with an undamped zero-speed component.
This case cannot be stabilized by the proportional boundary
feedback considered in \cite{HW16}, and the zero-speed component
does not satisfy the intrinsic stability assumptions used in
\cite{AVKK22,AVKK24}. We construct an explicit backstepping
boundary feedback that achieves asymptotic stability in $L^2$.

A key feature of our approach is the use of a nonstandard
integro-differential target system whose transport component vanishes in
finite time, after which the remaining zero-speed dynamics reduce to an
autonomous evolution equation governed by a Volterra integral operator.
Its asymptotic stability in $L^2$ is proved in two ways: by combining a
Lyapunov functional with LaSalle's invariance principle and by using an
abstract spectral semigroup argument. Together with the
explicitly solvable backstepping kernel system, this yields an explicit
boundary feedback law for the original system, which is convenient to
implement. In contrast to proportional control, the resulting feedback
\eqref{eq:explicit-feedback-law} provides a more powerful feedback
mechanism by incorporating weighted information from the full
distributed state.

The remainder of this paper is organized as follows.  Section 2
formulates the problem and states the main result.  Section 3 constructs
the backstepping transformation, solves the kernel equations, and derives
the explicit boundary feedback law.  Section 4 establishes the
well-posedness and asymptotic stability of the target system in $L^2$.
The stability is proved first by combining a Lyapunov functional with
LaSalle's invariance principle and then by applying an alternative abstract
semigroup approach.  Finally, Section 5 presents numerical simulations
illustrating the effectiveness of the proposed approach.

\section{Problem formulation and main results}

We consider the following coupled first-order hyperbolic system
with one zero characteristic speed:
\begin{equation}\label{eq:system}
\left\{
\begin{aligned}
u_t(t,x)-u_x(t,x)&=\beta v(t,x),\\
v_t(t,x)&=\gamma u(t,x),
\end{aligned}
\right.
\qquad (t,x)\in (0,\infty)\times(0,1),
\end{equation}
where $u(t,x)$ and $v(t,x)$ denote the system states, and
$\beta,\gamma \neq 0$ are given constants.

The system is supplemented with the initial conditions
\begin{equation}\label{eq:initial}
u(0,x)=u_0(x),\quad
v(0,x)=v_0(x),\quad x\in(0,1),
\end{equation}
and the boundary control
\begin{equation}\label{eq:boundary}
u(t,1)=U(t),\qquad t>0.
\end{equation}

Throughout this paper, we assume that
$(u_0,v_0)\in L^2(0,1)\times L^2(0,1)$.  Our objective is to design a boundary feedback
law $U=\mathcal{F}(u,v)$ such that the resulting closed-loop system is
asymptotically stable in $L^2(0,1)\times L^2(0,1)$.

For related well-posedness results and the obstruction to exact
null controllability in finite time, we refer to \cite{HW16}.

For $\beta\gamma<0$, Hu and Wang \cite{HW16} showed that a simple
proportional boundary feedback achieves asymptotic stabilization.
Therefore, we focus on the complementary case $\beta\gamma>0$.

\subsection{Main result}

\begin{theorem}\label{thm:main}
Assume that $\beta\gamma>0$ and let $U$ be the boundary feedback
given by \eqref{eq:explicit-feedback-law}. Then, for every
$(u_0,v_0)\in L^2(0,1)\times L^2(0,1)$, the closed-loop
initial-boundary value problem
\begin{equation}\label{eq:main-closed-loop}
\left\{
\begin{aligned}
u_t-u_x&=\beta v,
&& (t,x)\in(0,\infty)\times(0,1),\\
v_t&=\gamma u,
&& (t,x)\in(0,\infty)\times(0,1),\\
u(t,1)&=U(t),
&& t>0,\\
u(0,x)&=u_0(x),\quad v(0,x)=v_0(x),
&& x\in(0,1),
\end{aligned}
\right.
\end{equation}
admits a unique solution in the sense of distributions
\[
(u,v)\in C\bigl([0,\infty);
L^2(0,1)\times L^2(0,1)\bigr).
\]
Moreover, the closed-loop system is asymptotically stable in
$L^2(0,1)\times L^2(0,1)$. In particular,
\[
\lim_{t\to\infty}
\|(u(t,\cdot),v(t,\cdot))\|_{L^2(0,1)\times L^2(0,1)}
=0.
\]
\end{theorem}

\begin{corollary}\label{cor:arbitrary-speed}
Let $\alpha\neq0$ and assume that $\beta\gamma>0$. Consider the system
\[
\begin{cases}
u_t-\alpha u_x=\beta v,\\
v_t=\gamma u,
\end{cases}
\qquad (t,x)\in(0,\infty)\times(0,1).
\]
Then there exists an explicit backstepping control $U(t)$,
acting at the inflow boundary for $u$, namely,
\[
x=
\begin{cases}
1, & \text{ if }\alpha>0,\\
0, & \text{ if }\alpha<0,
\end{cases}
\]
such that, for every $(u_0,v_0)\in L^2(0,1)\times L^2(0,1)$, the corresponding
closed-loop problem admits a unique solution in the sense of distributions
\[
(u,v)\in C\bigl([0,\infty);L^2(0,1)\times L^2(0,1)\bigr),
\]
satisfying
\[
\lim_{t\to\infty}
\|(u(t,\cdot),v(t,\cdot))\|_{L^2(0,1)\times L^2(0,1)}=0.
\]
\end{corollary}

\begin{proof}
If $\alpha>0$, set
$\tau=\alpha t$.
Then the system becomes
$
\widetilde u_\tau-\widetilde u_x=\frac{\beta}{\alpha}\widetilde v,
$ $ 
\widetilde v_\tau=\frac{\gamma}{\alpha}\widetilde u,
$
and
$
\frac{\beta\gamma}{\alpha^2}>0.
$
Hence the conclusion follows from Theorem~\ref{thm:main}.

If $\alpha<0$, set
$\tau=-\alpha t$, $\xi=1-x.$
Then the system becomes
$
\widetilde u_\tau-\widetilde u_\xi=-\frac{\beta}{\alpha}\widetilde v,$ $
\widetilde v_\tau=-\frac{\gamma}{\alpha}\widetilde u,
$
with
$
\left(-\frac{\beta}{\alpha}\right)
\left(-\frac{\gamma}{\alpha}\right)
=\frac{\beta\gamma}{\alpha^2}>0.
$
Theorem~\ref{thm:main} therefore applies again, with the boundary
feedback acting at $x=0$.
\end{proof}

\begin{remark}
The resulting backstepping feedback
\eqref{eq:explicit-feedback-law} has a richer structure than
proportional control: rather than relying only on boundary values,
the boundary input depends on the state over the whole spatial domain through explicitly known integral kernels.
\end{remark}

\section{Backstepping design}

\subsection{Design of the target system}
The target system is designed so that it admits a convenient stability analysis while yielding kernel equations that are straightforward to solve.

Accordingly, we introduce the following target
system:
\begin{equation}\label{eq:target}
\left\{
\begin{aligned}
w_t-w_x&=0,
&& (t,x)\in(0,\infty)\times(0,1),\\
z_t&=\gamma w
-\beta\gamma\int_0^x z(t,y)\,\dd y,
&& (t,x)\in(0,\infty)\times(0,1),\\
w(t,1)&=0,
&& t>0,\\
w(0,x)&=w_0(x),\quad z(0,x)=z_0(x),
&& x\in(0,1).
\end{aligned}
\right.
\end{equation}
The well-posedness and asymptotic stability of the target problem are
established in Section~4 for initial data in $L^2(0,1)\times L^2(0,1)$.

\subsection{Backstepping transformation and kernel equations}

To map the original system \eqref{eq:system}--\eqref{eq:boundary} into
the target system \eqref{eq:target},
we seek a Volterra backstepping transformation of the form
\begin{equation}\label{eq:general-backstepping}
\begin{aligned}
w(t,x)
&=
u(t,x)
+\int_0^xK(x,y)u(t,y)\,\dd y
+\int_0^xL(x,y)v(t,y)\,\dd y,\\
z(t,x)
&=
v(t,x)
+\int_0^xM(x,y)u(t,y)\,\dd y
+\int_0^xN(x,y)v(t,y)\,\dd y,
\end{aligned}
\end{equation}
where $K,L,M,N$ are defined on
$\mathcal{T}:=\{(x,y):0\leq y\leq x\leq1\}$.

We first compute the dynamics of $w$.  Differentiating the first equation
in \eqref{eq:general-backstepping}, integrating by parts, and using the
original system and the transformation itself, we obtain
\begin{align*}
w_t-w_x
={}&
(\beta-L(x,x))z(t,x)
-K(x,0)w(t,0)
\\
&-\int_0^x
\bigl(K_x+K_y-\gamma L+(\beta-L(x,x))M\bigr)(x,y)
u(t,y)\,\dd y
\\
&-\int_0^x
(L_x-\beta K+(\beta-L(x,x))N)(x,y)
v(t,y)\,\dd y .
\end{align*}

Similarly, differentiating the second equation in
\eqref{eq:general-backstepping} gives
\begin{align*}
z_t
={}&
(\gamma+M(x,x))w(t,x)
-M(x,0)w(t,0)
\\
&+\int_0^x
\bigl(\gamma N-M_y-(\gamma+M(x,x))K\bigr)(x,y)
u(t,y)\,\dd y
\\
&+\int_0^x
\bigl(\beta M-(\gamma+M(x,x))L\bigr)(x,y)
v(t,y)\,\dd y .
\end{align*}

To recover the desired target system, we choose $M\equiv0$ and
$N\equiv K$.  Matching the $w$-equation with $w_t-w_x=0$ and the $z$-equation
with the target dynamics, except for the integral term,
yields the kernel equations
\begin{equation}\label{eq:kernel}
\left\{
\begin{aligned}
K_x(x,y)+K_y(x,y)-\gamma L(x,y)&=0,\\
L_x(x,y)-\beta K(x,y)&=0,
\end{aligned}
\right.
\qquad 0<y<x<1.
\end{equation}
together with the boundary conditions
\begin{equation}
\left\{
\begin{aligned}
K(x,0)&=0,
\\
L(x,x)&=\beta ,
\end{aligned}
\right.
\qquad 0\leq x\leq1 .
\label{eq:kernel-bc}
\end{equation}

Moreover, with this choice of $M$ and $N$, the equation for $z$ becomes
\[
z_t(t,x)
=
\gamma w(t,x)
-\gamma\int_0^xL(x,y)v(t,y)\,\dd y.
\]
The relations $L_x=\beta K$ and $L(y,y)=\beta$ imply that
\[
L(x,y)=\beta\left(1+\int_y^x K(s,y)\,\dd s\right).
\]
Hence, by changing the order
of integration and using the definition of $z$, we have
\[
\int_0^xL(x,y)v(t,y)\,\dd y
=
\beta\int_0^xz(t,y)\,\dd y.
\]
Thus, the transformed $z$-equation coincides with the second equation of
the target system \eqref{eq:target}.  The design therefore reduces to
solving \eqref{eq:kernel}--\eqref{eq:kernel-bc}.

\subsection{Solvability of kernel equations}

The characteristic domain $\mathcal T$ for the kernel equations is
illustrated in Figure~\ref{fig:kernel-characteristic-domain}.

\begin{figure}[htbp]
\centering
\includegraphics[width=0.72\textwidth]{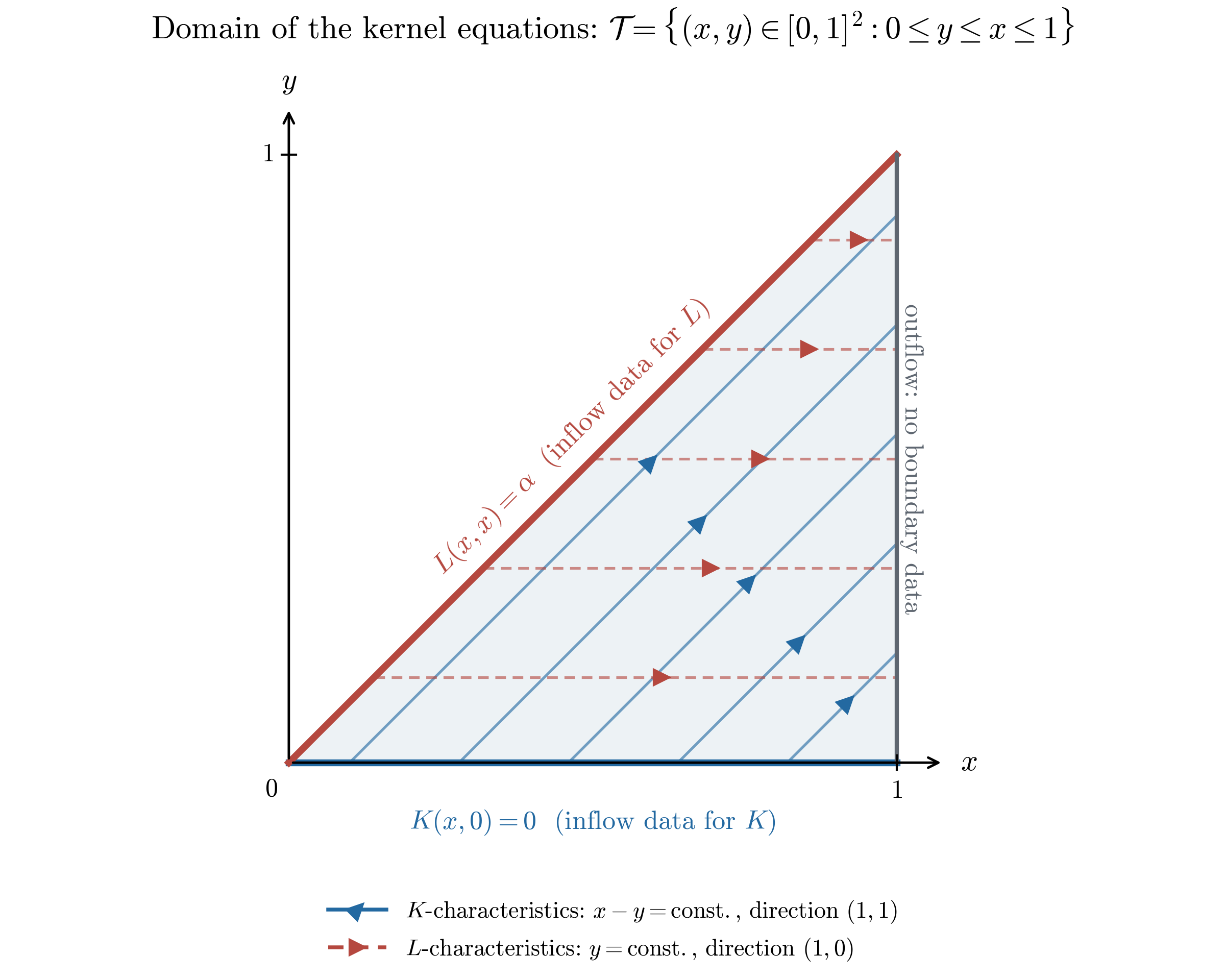}
\caption{Characteristic domain for the kernel equations.}
\label{fig:kernel-characteristic-domain}
\end{figure}

The following proposition provides an explicit solution to the kernel
system \eqref{eq:kernel}--\eqref{eq:kernel-bc}.

\begin{proposition}\label{prop:kernel-solution}
Assume that $\beta\gamma>0$.  The kernel system
\eqref{eq:kernel}--\eqref{eq:kernel-bc} admits a unique solution in
$\bigl(C^0(\mathcal T)\cap
C^\infty(\mathring{\mathcal T})\bigr)^2$, given by
\begin{equation}
\left\{
\begin{aligned}
K(x,y)
&=
\frac{\sqrt{\beta\gamma y}}{\sqrt{x-y}}
I_1
\left(
2\sqrt{\beta\gamma y(x-y)}
\right)
,
\\[2mm]
L(x,y)
&=
\beta
I_0
\left(
2\sqrt{\beta\gamma y(x-y)}
\right),
\end{aligned}
\right.
\label{eq:kernel-solution}
\end{equation}
where $I_0$ and $I_1$ denote the modified Bessel functions of the first
kind of orders zero and one, respectively; see, e.g., \cite{Wat44,NIST10}. The expression for $K$ is
understood through its continuous extension to $x=y$, for which
$K(x,x)=\beta\gamma x$.
\end{proposition}

\begin{proof}
From the second
equation in \eqref{eq:kernel}, we have $K=\beta^{-1}L_x$.
Substituting this expression into the first equation gives
\begin{equation}
L_{xx}+L_{xy}-\beta\gamma L=0.
\label{eq:L-second-order}
\end{equation}

To rewrite \eqref{eq:L-second-order} in its canonical Goursat form, we
introduce $\xi=x-y$, $\eta=y$, and
$\widetilde L(\xi,\eta):=L(\xi+\eta,\eta)$.  Since
$\partial_x=\partial_\xi$ and
$\partial_y=-\partial_\xi+\partial_\eta$, we have
$L_{xx}+L_{xy}=(\partial_x+\partial_y)\partial_x L
=\partial_\eta\partial_\xi\widetilde L
=\widetilde L_{\xi\eta}$.  Consequently,
\eqref{eq:L-second-order} becomes
\begin{equation}
\widetilde L_{\xi\eta}=\beta\gamma\widetilde L.
\label{eq:L-Goursat}
\end{equation}
The boundary conditions \eqref{eq:kernel-bc} become
\[
\widetilde L(0,\eta)=\beta,
\qquad \widetilde L_\xi(\xi,0)=0.
\]
Since $\widetilde L(0,0)=\beta$, the second condition gives
$\widetilde L(\xi,0)=\beta$. Thus, the boundary conditions for the Goursat problem are
\[
\widetilde L(0,\eta)=\beta,
\qquad \widetilde L(\xi,0)=\beta,
\qquad 0\leq\xi,\eta\leq1.
\]

The reduction below follows the classical construction of the Riemann
function for constant-coefficient hyperbolic equations.  See
\cite[Chapter V, Section 5, pp.~458--459]{CH89}.

Equation~\eqref{eq:L-Goursat} is invariant under the scaling
$(\xi,\eta)\mapsto(\lambda\xi,\lambda^{-1}\eta)$ for $\lambda>0$,
which preserves the product $\xi\eta$. This motivates seeking a
solution depending only on this product.
Accordingly, we set $s=\xi\eta$ and seek a solution of the form
$\widetilde L(\xi,\eta)=\beta\phi(s)$.  Since
$s_\xi=\eta$ and $s_\eta=\xi$, we have
\[
\widetilde L_{\xi\eta}
=
\beta\bigl(\phi'(s)+s\phi''(s)\bigr).
\]
Moreover, $\widetilde L_\xi=\beta\eta\phi'(s)$, so the boundary
condition $\widetilde L_\xi(\xi,0)=0$ is automatically satisfied.
Substituting this expression directly into \eqref{eq:L-Goursat} and
dividing by $\beta\neq0$, we obtain the ordinary differential equation
\begin{equation}
s\phi''(s)+\phi'(s)-\beta\gamma\phi(s)=0.
\label{eq:phi-equation}
\end{equation}

Next, introduce $r=2\sqrt{\beta\gamma s}$ and let
$\psi(r)=\phi(s)$.  Using
$\dd r/\dd s=2\beta\gamma/r$, equation \eqref{eq:phi-equation} becomes
\[
r^2\psi''(r)
+r\psi'(r)
-r^2\psi(r)=0 .
\]

This is the modified Bessel equation of order zero; see, e.g.,
\cite{Wat44,NIST10}. Therefore,
\[
\psi(r)
=
c_1 I_0(r)+c_2 K_0(r),
\]
where $I_0$ and $K_0$ denote the modified Bessel functions of the
first and second kinds, respectively.

Since $K_0(r)$ is singular at $r=0$, the regularity of the kernel along
$\xi=0$ implies $c_2=0$.  The boundary condition
$\widetilde L(0,\eta)=\beta$ corresponds to $s=r=0$.  Since $I_0(0)=1$,
we obtain $c_1=1$.

Returning to the original variables $\xi=x-y$ and $\eta=y$, we obtain
\[
L(x,y)
=
\beta
I_0
\left(
2\sqrt{\beta\gamma y(x-y)}
\right).
\]

Finally, from $K(x,y)=\beta^{-1}L_x(x,y)$ and
$I_0'(r)=I_1(r)$, we obtain
\[
K(x,y)
=
\frac{\sqrt{\beta\gamma y}}{\sqrt{x-y}}
I_1
\left(
2\sqrt{\beta\gamma y(x-y)}
\right)
.
\]

It remains to verify the boundary conditions. Since $I_0(0)=1$, we have
$L(x,x)=\beta$.

Moreover, using the asymptotic property
\[
I_1(r)\sim \frac r2,\qquad \text{as } r\to0,
\]
and taking $r=2\sqrt{\beta\gamma y(x-y)}$, we obtain
\[
K(x,y)
\sim
\frac{\sqrt{\beta\gamma y}}{\sqrt{x-y}}
\sqrt{\beta\gamma y(x-y)}
=\beta\gamma y,
\qquad \text{as } x\to y+.
\]
Hence, $K$ admits a continuous extension to the diagonal by setting
\[
K(y,y)=\beta\gamma y,\qquad y>0.
\]
On the other hand, the explicit formula gives $K(x,0)=0$ for $x>0$.
The two extensions agree at the corner $(0,0)$, where $K(0,0)=0$.

Therefore, the obtained kernels satisfy the kernel equations
\eqref{eq:kernel} with the boundary conditions \eqref{eq:kernel-bc}.

Moreover, the analyticity of $I_0$ and $I_1$
implies that $K,L\in C^\infty(\mathring{\mathcal T})$. 

It remains to prove that the solution to
\eqref{eq:kernel}--\eqref{eq:kernel-bc} is unique.
Let $(K_1,L_1)$ and $(K_2,L_2)$ be two solutions, and set
$\widehat K=K_1-K_2$ and $\widehat L=L_1-L_2$.
Then $(\widehat K,\widehat L)$ satisfies \eqref{eq:kernel} with the
homogeneous boundary conditions
$\widehat K(x,0)=0$ and $\widehat L(x,x)=0$.

Applying the same change of variables as above gives
\[
\widetilde{\widehat L}_{\xi\eta}
=\beta\gamma\widetilde{\widehat L},
\qquad
\widetilde{\widehat L}(0,\eta)=0.
\]
Moreover, since $\widehat L_x=\beta\widehat K$ in the interior and
$\widehat K\in C(\mathcal T)$, the relation extends continuously to
the boundary. Hence
\[
\partial_\xi\widetilde{\widehat L}(\xi,0)
=\beta\widehat K(\xi,0)=0.
\]
Since $\widetilde{\widehat L}(0,0)=0$, the last condition implies
$\widetilde{\widehat L}(\xi,0)=0$. Hence,
\[
\widetilde{\widehat L}(\xi,\eta)
=
\beta\gamma
\int_0^\xi\int_0^\eta
\widetilde{\widehat L}(p,q)\,\mathrm dq\,\mathrm dp.
\]
Iterating this identity yields
\[
|\widetilde{\widehat L}(\xi,\eta)|
\le
\|\widetilde{\widehat L}\|_\infty
\frac{(|\beta\gamma|\xi\eta)^n}{(n!)^2},
\qquad n\ge1.
\]
Letting $n\to\infty$ gives
$\widetilde{\widehat L}\equiv0$, and hence $\widehat L=0$.
Finally, $\widehat K=\beta^{-1}\widehat L_x=0$.
Thus the solution is unique, which concludes the proof of Proposition~\ref{prop:kernel-solution}.
\end{proof}

\subsection{Norm equivalence and explicit feedback law}

With $M\equiv0$ and $N=K$, the general transformation
\eqref{eq:general-backstepping} becomes
\begin{equation}\label{eq:backstepping-transformation}
\begin{aligned}
w(t,x)
&=
u(t,x)
+\int_0^xK(x,y)u(t,y)\,\dd y
+\int_0^xL(x,y)v(t,y)\,\dd y,\\
z(t,x)
&=
v(t,x)
+\int_0^xK(x,y)v(t,y)\,\dd y,
\end{aligned}
\end{equation}
where $K$ and $L$ are given explicitly by
Proposition~\ref{prop:kernel-solution}.

\begin{proposition}\label{prop:transformation}
The transformation \eqref{eq:backstepping-transformation} defines a
bounded linear isomorphism on $L^2(0,1)\times L^2(0,1)$.  Consequently,
there exists a constant $C>0$, independent of $t$, such that
\[
C^{-1}\|(u,v)(t,\cdot)\|_{L^2\times L^2}
\leq \|(w,z)(t,\cdot)\|_{L^2\times L^2}
\leq C\|(u,v)(t,\cdot)\|_{L^2\times L^2}.
\]
\end{proposition}

\begin{proof}
Since $K,L\in C^0(\mathcal T)$, the transformation
\eqref{eq:backstepping-transformation} is bounded on $L^2(0,1)\times L^2(0,1)$.
By the standard theory of Volterra integral equations of the second kind,
it is invertible, and its inverse is again a bounded Volterra integral
transformation with continuous kernels. Hence, the transformation is a
bounded linear isomorphism on $L^2(0,1)\times L^2(0,1)$, and the asserted norm
equivalence follows.
\end{proof}

Evaluating the first equation of
\eqref{eq:backstepping-transformation}
at \(x=1\), and using $u(t,1)=U(t)$ and $w(t,1)=0$, we obtain
\[
U(t)
=
-
\int_0^1 K(1,y)u(t,y)\,dy
-
\int_0^1 L(1,y)v(t,y)\,dy.
\]
Substituting the explicit expressions of the kernels gives
\begin{align}
U(t)
={}&
-\sqrt{\beta\gamma}
\int_0^1
\sqrt{\frac{y}{1-y}}\,
I_1\!\left(
2\sqrt{\beta\gamma\,y(1-y)}
\right)u(t,y)\,dy
\nonumber\\
&\quad
-\beta
\int_0^1
I_0\!\left(
2\sqrt{\beta\gamma\,y(1-y)}
\right)v(t,y)\,dy.
\label{eq:explicit-feedback-law}
\end{align}
The apparent singularity at $y=1$ is removable, with
$K(1,1)=\beta\gamma$. Since both kernels are continuous,
this formula defines a bounded linear full-state feedback
on $L^2(0,1)\times L^2(0,1)$.

\section{Stability of the target system and proof of Theorem~\ref{thm:main}}
This section is organized as follows. In
Subsection~\ref{subsec:L2-wellposedness}, we establish the well-posedness
of the target system for $L^2$ initial data. In
Subsection~\ref{subsec:L2-asymptotic-stability}, we prove its asymptotic
stability by reducing the long-time dynamics to the zero-speed component
and applying the Lyapunov--LaSalle method. In
Subsection~\ref{subsec:proof-main-theorem}, we complete the proof of
Theorem~\ref{thm:main} by transferring the well-posedness and stability
results back to the original closed-loop system. Finally, in
Subsection~\ref{subsec:alternative-strong-stability}, we provide an
alternative proof of the stability of the reduced system based on the
Arendt--Batty--Lyubich--V\~{u} theorem.

Throughout this section, we assume that $\beta\gamma>0$.

\label{sec:L2-zero-input}

\subsection{\texorpdfstring{$L^2$}{L2} well-posedness}
\label{subsec:L2-wellposedness}

We first establish the well-posedness of the target system for
$L^2$ initial data.  We use the notion of solution in the sense of
distributions adopted by Hu and Wang in
\cite[Definition~2.1]{HW16}.

\begin{proposition}\label{prop:L2-wellposed-W0}
For every $(w_0,z_0)\in L^2(0,1)\times L^2(0,1)$, the target system
\begin{equation}\label{eq:L2-target}
\left\{
\begin{aligned}
w_t-w_x&=0,
&& (t,x)\in(0,\infty)\times(0,1),\\
z_t&=\gamma w
-\beta\gamma\int_0^x z(t,y)\,\dd y,
&& (t,x)\in(0,\infty)\times(0,1),\\
w(t,1)&=0,
&& t>0,\\
w(0,x)&=w_0(x),\quad z(0,x)=z_0(x),
&& x\in(0,1).
\end{aligned}
\right.
\end{equation}
admits a unique solution in the sense of distributions satisfying
\[
w\in C([0,\infty);L^2(0,1)),\quad
z\in C^1([0,\infty);L^2(0,1)).
\]
\end{proposition}

\begin{proof}
We first consider the transport equation $w_t-w_x=0$ with the boundary
condition $w(t,1)=0$.  By the method of characteristics, its solution is
given by
\begin{equation}\label{eq:W0-characteristics}
w(t,x)=
\begin{cases}
w_0(x+t),&x+t\leq1,\\
0,&x+t>1.
\end{cases}
\end{equation}
We have
\[
w\in C([0,\infty);L^2(0,1)).
\]

The characteristic formula is the unique solution of the transport equation in the sense of distributions.

Next, we consider the equation for $z$.  Define the classical Volterra
integration operator by
\[
(\mathcal V\varphi)(x):=\int_0^x\varphi(y)\,\dd y,
\qquad \varphi\in L^2(0,1).
\]
For every $\varphi\in L^2(0,1)$, the Cauchy--Schwarz inequality gives
\[
|(\mathcal V\varphi)(x)|^2
\leq x\int_0^x|\varphi(y)|^2\,\dd y.
\]
After integration in $x$, we obtain
\[
\|\mathcal V\varphi\|_{L^2(0,1)}
\leq\frac1{\sqrt2}\|\varphi\|_{L^2(0,1)}.
\]
Thus, $\mathcal V:L^2(0,1)\to L^2(0,1)$ is a bounded linear operator.

The second equation of \eqref{eq:L2-target} can now be viewed as
an ordinary differential equation in the Banach space $L^2(0,1)$:
\[
z_t(t)=\gamma w(t)-\beta\gamma\mathcal Vz(t),
\qquad z(0)=z_0.
\]
Since $w\in C([0,\infty);L^2(0,1))$ and $\mathcal V$ is bounded on
$L^2(0,1)$, the right-hand side is continuous in $t$ and globally
Lipschitz with respect to $z$.  The classical existence and uniqueness
theorem for ordinary differential equations in Banach spaces therefore
yields $z\in C^1([0,\infty);L^2(0,1))$.  The Banach-space equation for
$z$ holds in
the distributional sense, and its uniqueness gives the uniqueness of the
distributional solution.
\end{proof}

\subsection{\texorpdfstring{$L^2$}{L2} asymptotic stability}
\label{subsec:L2-asymptotic-stability}

In this subsection, we establish the asymptotic stability of the target
system in $L^2(0,1)\times L^2(0,1)$. The main result is stated below.

\begin{theorem}\label{thm:L2-target-W0-stability}
Let $(w,z)$ be the solution of the target system
\eqref{eq:L2-target} associated with initial data
$(w_0,z_0)\in L^2(0,1)\times L^2(0,1)$. Then
\[
\lim_{t\to\infty}
\|(w(t,\cdot),z(t,\cdot))\|_{L^2(0,1)\times L^2(0,1)}
=0.
\]
\end{theorem}

\begin{proof}
The proof consists of two steps.

\textit{Step 1: Finite-time extinction of the transport component.}

Since $w(t,1)=0$, the characteristics of $w_t-w_x=0$ give
\[
w(t,x)=
\begin{cases}
w_0(x+t),& x+t\leq1,\\[1mm]
0,& x+t>1.
\end{cases}
\]
Hence
\[
w(t,\cdot)\equiv0,
\qquad t\geq1.
\]

\textit{Step 2: Stability of the reduced system.}

For $t\geq1$, the dynamics of $z$ reduce to
\[
z_t(t,x)
=
-\beta\gamma\int_0^x z(t,y)\,\dd y.
\]
By Proposition~\ref{prop:L2-reduced-stability}, which will be stated
and proved below,
\[
\|z(t,\cdot)\|_{L^2(0,1)}
\longrightarrow 0
\qquad\text{as }t\to\infty.
\]
Together with $w(t,\cdot)\equiv0$ for $t\geq1$, this proves the desired
stability of the target system.
\end{proof}

To analyze the stability of the reduced system for $z$, we first
establish the well-posedness of the corresponding Cauchy problem.

For convenience of presentation, we shift the reduced system in time
by one unit and take
\[
\widetilde z_0:=z(1,\cdot)\in L^2(0,1)
\]
as the initial state for the following discussion. After this time
translation, we keep the notation $z(t,\cdot)$ for the translated
trajectory.

In addition to $L^2(0,1)$, we also establish well-posedness in $\mathcal H_{(0)}^1(0,1)$, which will be needed for the smooth approximation argument below.
\begin{proposition}\label{prop:L2-reduced-wellposedness}
Let
$
X=L^2(0,1)
$ or $
X=\mathcal H_{(0)}^1(0,1)
:=
\{\phi\in H^1(0,1):\phi(0)=0\}.
$
For any initial condition $\widetilde z_0\in X$, the system
\begin{equation}
\begin{cases}
z_t(t,x)
=
-\beta\gamma
\displaystyle\int_0^x z(t,y)\,\dd y,
\\[2mm]
z(0,x)=\widetilde z_0(x),
\end{cases}
\label{eq:L2-reduced-system}
\end{equation}
admits a unique solution satisfying
\[
z\in C^\infty([0,\infty);X).
\]
\end{proposition}

\begin{proof}
Let $\mathcal V$ denote the Volterra operator
\[
(\mathcal V\phi)(x)
=
\int_0^x\phi(y)\,\dd y.
\]
The operator $\mathcal V$ is bounded on both $L^2(0,1)$ and
$\mathcal H_{(0)}^1(0,1)$. Hence $-\beta\gamma\mathcal V$ generates
a uniformly continuous $C_0$-semigroup on either space. Therefore,
for every $\widetilde z_0\in X$, the reduced system admits a unique
solution
\[
z(t)=e^{-\beta\gamma t\mathcal V}\widetilde z_0,
\]
which satisfies
\[
z\in C^\infty([0,\infty);X).
\]
\end{proof}

\begin{proposition}\label{prop:L2-reduced-stability}
For every initial condition $\widetilde z_0\in L^2(0,1)$, the unique
solution of \eqref{eq:L2-reduced-system} satisfies
\[
\lim_{t\rightarrow\infty}
\|z(t,\cdot)\|_{L^2(0,1)}
=0.
\]
\end{proposition}

The proof of Proposition~\ref{prop:L2-reduced-stability} is based on the LaSalle invariance principle, together with
Lyapunov energy estimates and a smooth approximation argument. We first
recall the following infinite-dimensional version of
LaSalle's invariance principle, which extends the classical result of
LaSalle \cite{L60}; see, e.g., \cite[Theorems~9.1.8 and~9.2.3]{CH98}.

\begin{lemma}[LaSalle's invariance principle]
Let $X$ be a Banach space and $(S(t))_{t\geq0}$ a $C_0$-semigroup on $X$.
Let $V:X\to\mathbb R$ be a continuous Lyapunov functional, non-increasing along the trajectories of $S(t)$.
Assume that
$
\dot V(x):=
\left.\frac{\mathrm d}{\mathrm dt}V(S(t)x)\right|_{t=0^+}
$
exists for every $x\in X$ and defines a continuous function on $X$.
Define $E:=\{x\in X:\dot V(x)=0\}$, and let
$
M:=\{x\in E:S(t)x\in E,\ \forall t\geq0\}
$
be the largest positively invariant subset of $E$.

If, for some $x_0\in X$, the orbit $\{S(t)x_0:t\geq0\}$ is precompact
in $X$, then
\[
\operatorname{dist}(S(t)x_0,M)\to0,\qquad t\to\infty.
\]
\end{lemma}

\begin{proof}[Proof of Proposition~\ref{prop:L2-reduced-stability}]
Consider the Lyapunov functional
\begin{equation}\label{eq:L2-Lyapunov-functional}
V(z)
:=
\frac12\int_0^1 z^2(x)\,\dd x .
\end{equation}

Along the solutions of \eqref{eq:L2-reduced-system}, we have
\begin{align}
\frac{\dd}{\dd t}V(z(t))
&=
\int_0^1 z(t,x)z_t(t,x)\,\dd x
\nonumber\\
&=
-\beta\gamma
\int_0^1
z(t,x)
\left(
\int_0^x z(t,y)\,\dd y
\right)\,\dd x
\nonumber\\
&=
-\frac{\beta\gamma}{2}
\left(
\int_0^1 z(t,x)\,\dd x
\right)^2
\leq0 .
\label{eq:L2-V}
\end{align}

We first consider initial data in the dense subspace
\[
\mathcal D:=\mathcal H_{(0)}^1(0,1)
\subset L^2(0,1).
\]
Let $\widetilde z_0\in\mathcal D$.  By Proposition~\ref{prop:L2-reduced-wellposedness}, the corresponding solution
remains in $\mathcal H_{(0)}^1(0,1)$ and hence satisfies $z(t,0)=0$.
Moreover, differentiating the reduced equation with respect to $x$ gives
$z_{xt}=-\beta\gamma z$.  Therefore,
\[
\begin{aligned}
\frac{\dd}{\dd t}\frac12\|z_x(t,\cdot)\|_{L^2(0,1)}^2
&=
-\beta\gamma\int_0^1z_x(t,x)z(t,x)\,\dd x\\
&=
-\frac{\beta\gamma}{2}z^2(t,1)
\leq0.
\end{aligned}
\]
Together with $V(z(t))\leq V(\widetilde z_0)$, this yields
\[
\sup_{t\geq0}\|z(t,\cdot)\|_{H^1(0,1)}<\infty .
\]
The compact embedding $H^1(0,1)\hookrightarrow L^2(0,1)$ therefore
implies that the orbit $\{z(t,\cdot):t\geq0\}$ is precompact in
$L^2(0,1)$.

Hence, all the assumptions of the LaSalle invariance principle are
satisfied. Consequently,
\[
\operatorname{dist}_{L^2}(z(t,\cdot),M)\to0
\qquad\text{as }t\to\infty,
\]
where $M$ is the largest positively invariant subset of $E$.

By Proposition~\ref{prop:L2-largest-invariant-set}, which will be stated
and proved below, this largest
positively invariant subset is $M=\{0\}$. Therefore,
\[
\lim_{t\rightarrow\infty}
\|z(t,\cdot)\|_{L^2(0,1)}
=0 .
\]
This proves the result for
$\widetilde z_0\in\mathcal D$.

It remains to extend the conclusion to arbitrary
$\widetilde z_0\in L^2(0,1)$.
Let $(S(t))_{t\geq0}$ denote the semigroup of the reduced system. Applying the
Lyapunov estimate \eqref{eq:L2-V} to the difference of two solutions
gives
\[
\|S(t)\phi-S(t)\psi\|_{L^2(0,1)}
\leq
\|\phi-\psi\|_{L^2(0,1)},
\qquad t\geq0.
\]

Choose $\widetilde z_0^n\in\mathcal D$ such that
$\widetilde z_0^n\to\widetilde z_0$ in $L^2(0,1)$.
For any fixed $n$, the triangle inequality and the contractivity of $S(t)$
on $L^2(0,1)$ yield
\begin{align*}
\|S(t)\widetilde z_0\|_{L^2}
&\leq
\|S(t)(\widetilde z_0-\widetilde z_0^n)\|_{L^2}
+
\|S(t)\widetilde z_0^n\|_{L^2}
\\
&\leq
\|\widetilde z_0-\widetilde z_0^n\|_{L^2}
+
\|S(t)\widetilde z_0^n\|_{L^2}.
\end{align*}
Taking the limit superior on both sides, we obtain
\[
\limsup_{t\to\infty}\|S(t)\widetilde z_0\|_{L^2}
\leq
\|\widetilde z_0-\widetilde z_0^n\|_{L^2}.
\]
Since $n$ is arbitrary, letting $n\to\infty$ yields
\[
\lim_{t\to\infty}\|S(t)\widetilde z_0\|_{L^2}=0,
\]
for every $\widetilde z_0\in L^2(0,1)$.
This concludes the proof of Proposition~\ref{prop:L2-reduced-stability}.
\end{proof}

We now characterize the largest positively invariant subset used in the
preceding proof.

\begin{proposition}\label{prop:L2-largest-invariant-set}
The largest positively invariant subset of
\[
E
=
\left\{
z\in L^2(0,1):
\dot V(z)=0
\right\}
\]
is given by $M=\{0\}$, where $V$ is defined in
\eqref{eq:L2-Lyapunov-functional}.
\end{proposition}

\begin{proof}
Let $M$ denote the largest positively invariant subset of $E$.
Take any initial condition $\widetilde z_0\in M$.

Since $M$ is positively invariant, we have
$z(t,\cdot)\in M\subset E$ for all $t\geq0$.

Therefore, $\dot V(t)=0$ for $t\geq0$.  From
\eqref{eq:L2-V}, we obtain
\[
\int_0^1 z(t,x)\,\dd x=0,
\qquad t\geq0 .
\]

Differentiating with respect to time gives
\[
\int_0^1 z_t(t,x)\,\dd x=0.
\]

Using the equation
$z_t(t,x)=-\beta\gamma\int_0^x z(t,y)\,\dd y$, we obtain
\[
\int_0^1
\left(
\int_0^x z(t,y)\,\dd y
\right)\,\dd x
=0 .
\]

Changing the order of integration, we obtain
\begin{align*}
\int_0^1
\left(
\int_0^x z(t,y)\,\dd y
\right)\,\dd x
&=
\int_0^1
\left(
\int_y^1 z(t,y)\,\dd x
\right)\,\dd y
\\
&=
\int_0^1
(1-y)z(t,y)\,\dd y .
\end{align*}

Repeating the above argument, we obtain
\[
\int_0^1(1-x)^n z(t,x)\,\dd x=0,
\qquad n\geq0 .
\]

Hence, $z(t,\cdot)$ is orthogonal to every polynomial on $(0,1)$. Since polynomials are dense in $L^2(0,1)$, we conclude that
$z(t,\cdot)=0$ in $L^2(0,1)$ for every $t>0$. Since
$z\in C([0,\infty);L^2(0,1))$, letting $t\to0^+$ yields
$\widetilde z_0=0$.

Thus \(M\subset\{0\}\). Conversely, \(0\in E\), and the zero solution is positively invariant. Hence \(\{0\}\subset M\), and therefore \(M=\{0\}\).
\end{proof}

\subsection{Proof of Theorem~\ref{thm:main}}
\label{subsec:proof-main-theorem}
We now return to the original closed-loop system and complete the proof of the main result.

Combining Theorem~\ref{thm:L2-target-W0-stability} with
Propositions~\ref{prop:L2-wellposed-W0} and
\ref{prop:transformation}, we obtain the well-posedness and asymptotic
stability of the original closed-loop system. More precisely, under
the feedback law \eqref{eq:explicit-feedback-law}, for every
$(u_0,v_0)\in L^2(0,1)\times L^2(0,1)$, the closed-loop system
admits a unique solution
\[
(u,v)\in C([0,\infty);L^2(0,1)\times L^2(0,1)),
\]
with the equations and boundary condition understood in the sense of
distributions, and
\[
\lim_{t\to\infty}
\|(u(t),v(t))\|_{L^2(0,1)\times L^2(0,1)}=0.
\]
This completes the proof of Theorem~\ref{thm:main}.

\subsection{Alternative proof via an abstract strong stability criterion}
\label{subsec:alternative-strong-stability}

The stability of the reduced system \eqref{eq:L2-reduced-system} can
also be established directly by an abstract operator argument.  We work
on the complexification
$X=L^2(0,1;\mathbb C)$ of the real space $L^2(0,1)$.  Define the Volterra operator
$\mathcal V\in\mathcal L(X)$ by
\[
(\mathcal Vz)(x):=\int_0^x z(y)\,\dd y,
\]
and set
\[
A:=-\beta\gamma\mathcal V,
\qquad D(A):=X.
\]
By Proposition~\ref{prop:L2-reduced-wellposedness}, $A$ generates the semigroup
$(S(t))_{t\geq0}$.
The real subspace $L^2(0,1)$ is invariant under this semigroup; hence,
strong stability on $X$ implies the desired strong stability for
real-valued initial data.

We recall the following Arendt--Batty--Lyubich--V\~u theorem \cite{AB88,LP88}.

\begin{lemma}[Arendt--Batty--Lyubich--V\~u theorem]
Let $(S(t))_{t\geq0}$ be a bounded $C_0$-semigroup on a complex
reflexive Banach space $X$, with generator $A$. Assume that
\[
\sigma(A)\cap i\mathbb R \quad\text{is countable},
\qquad
\sigma_p(A)\cap i\mathbb R=\varnothing.
\]
Then $(S(t))_{t\geq0}$ is strongly stable, that is,
\[
\lim_{t\to\infty}\|S(t)x\|_X=0,
\qquad x\in X.
\]
\end{lemma}

\begin{proof}[Alternative proof of Proposition~\ref{prop:L2-reduced-stability}]
\leavevmode\par
Consider the complexification $X=L^2(0,1;\mathbb C)$. By the
complexified form of the energy identity \eqref{eq:L2-V}, using the
complex inner product and taking real parts, the semigroup generated
by $A$ is contractive and hence bounded.

For every $n\geq1$, the iterates of the Volterra operator satisfy
\[
(\mathcal V^n\phi)(x)
=
\frac{1}{(n-1)!}
\int_0^x(x-y)^{n-1}\phi(y)\,\dd y,
\]
so Young's convolution inequality gives
\[
\|\mathcal V^n\|
\leq
\int_0^1\frac{s^{n-1}}{(n-1)!}\,\dd s
=\frac{1}{n!}.
\]
By the spectral radius formula,
\[
r(\mathcal V)
=\lim_{n\to\infty}\|\mathcal V^n\|^{1/n}
\leq
\lim_{n\to\infty}\left(\frac1{n!}\right)^{1/n}
=0.
\]
Hence $r(\mathcal V)=0$. Since the spectrum of a bounded operator on a
complex Banach space is nonempty, it follows that
\[
\sigma(\mathcal V)=\{0\}.
\]
Since $A=-\beta\gamma\mathcal V$, we obtain
\[
\sigma(A)=\{0\},
\]
and hence $\sigma(A)\cap i\mathbb R=\{0\}$ is countable.

It remains to exclude $0$ as an eigenvalue of $A$. If $A\phi=0$, then
$\mathcal V\phi=0$, that is,
\[
\int_0^x\phi(y)\,\dd y=0
\quad\text{for a.e. }x\in(0,1).
\]
Since the left-hand side has an absolutely continuous representative,
it vanishes identically, and differentiation yields $\phi=0$ almost
everywhere. Thus
\[
\sigma_p(A)\cap i\mathbb R=\varnothing.
\]

The space $L^2(0,1;\mathbb C)$ is reflexive. Hence, the
Arendt--Batty--Lyubich--V{\~u} theorem yields
\[
\lim_{t\to\infty}\|S(t)z_0\|_{L^2(0,1)}=0,
\qquad z_0\in L^2(0,1;\mathbb C).
\]
Restricting to the invariant real subspace gives the asymptotic
stability asserted in Proposition~\ref{prop:L2-reduced-stability}.
\end{proof}

\section{Numerical simulations}

We illustrate the behavior of the uncontrolled system and compare it
with proportional and backstepping boundary feedback.  In all simulations,
we take $\beta=\gamma=1$, with the initial conditions
\[
u_0(x)=\cos(\pi x)+x^2,
\qquad
v_0(x)=\sin(\pi x)+x-0.6.
\]
We employ a method-of-lines discretization, using a second-order upwind
finite difference in space and the classical fourth-order Runge--Kutta
method in time.

For comparison, we use the $L^2$-energy
\[
E_{L^2}(t)
:=\|u(t,\cdot)\|_{L^2(0,1)}^2
+\|v(t,\cdot)\|_{L^2(0,1)}^2.
\]

We first consider the zero boundary input $U(t)=0$.  As shown in
Figure~\ref{fig:zero-input-states}, both state components grow in time.
The left panel of Figure~\ref{fig:unstable-energy-comparison} shows that,
after a short transient, the $L^2$-energy increases substantially
over the simulated time interval. This behavior illustrates the failure of the zero boundary input
to stabilize the system.

\begin{figure}[htbp]
\centering
\includegraphics[width=0.98\textwidth]{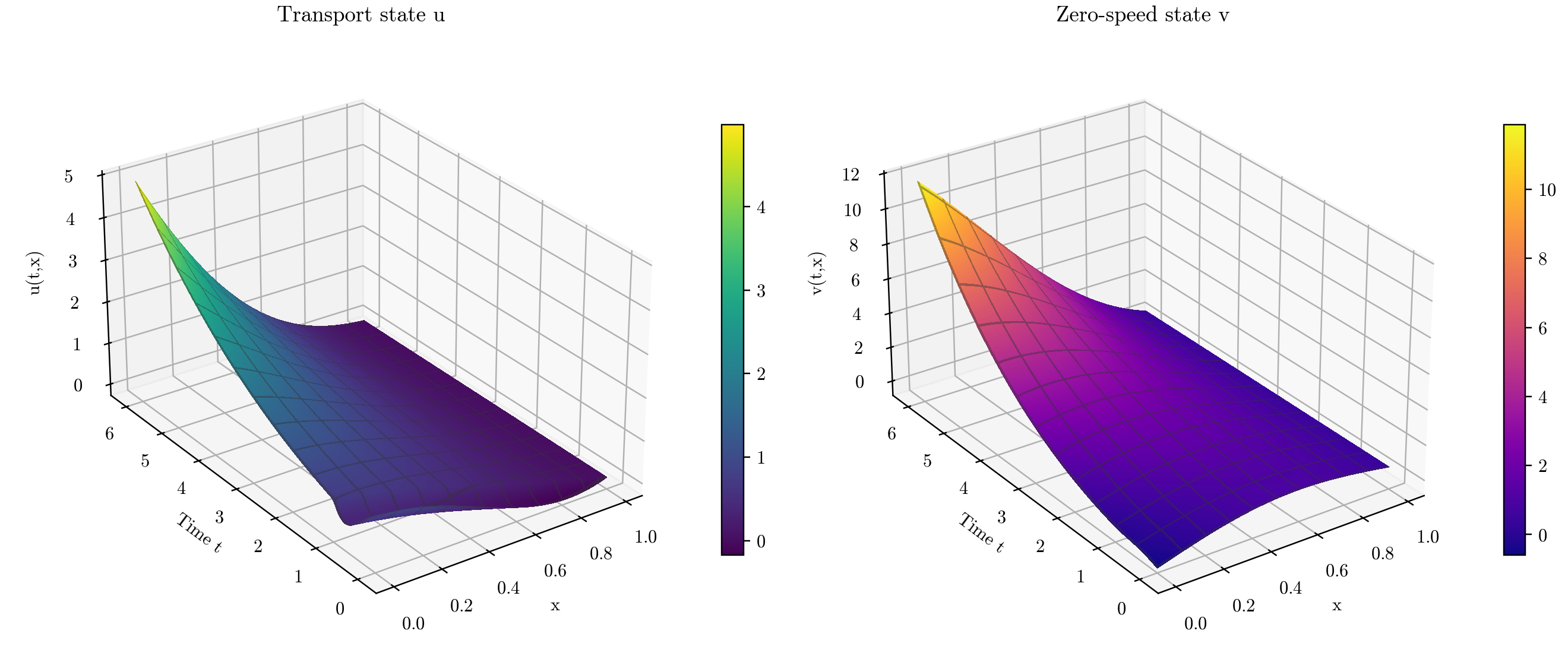}
\caption{Open-loop state evolution with $U(t)=0$.  Left: transport state
$u$.  Right: zero-speed state $v$.}
\label{fig:zero-input-states}
\end{figure}

Next, we apply the proportional boundary feedback
$U(t)=0.5u(t,0)$.  Figure~\ref{fig:proportional-feedback-states} and the
right panel of Figure~\ref{fig:unstable-energy-comparison} show that the
state amplitudes and the corresponding $L^2$-energy continue to
grow.  In particular, this proportional feedback does not stabilize the
system for the chosen initial condition.

\begin{figure}[htbp]
\centering
\includegraphics[width=0.98\textwidth]{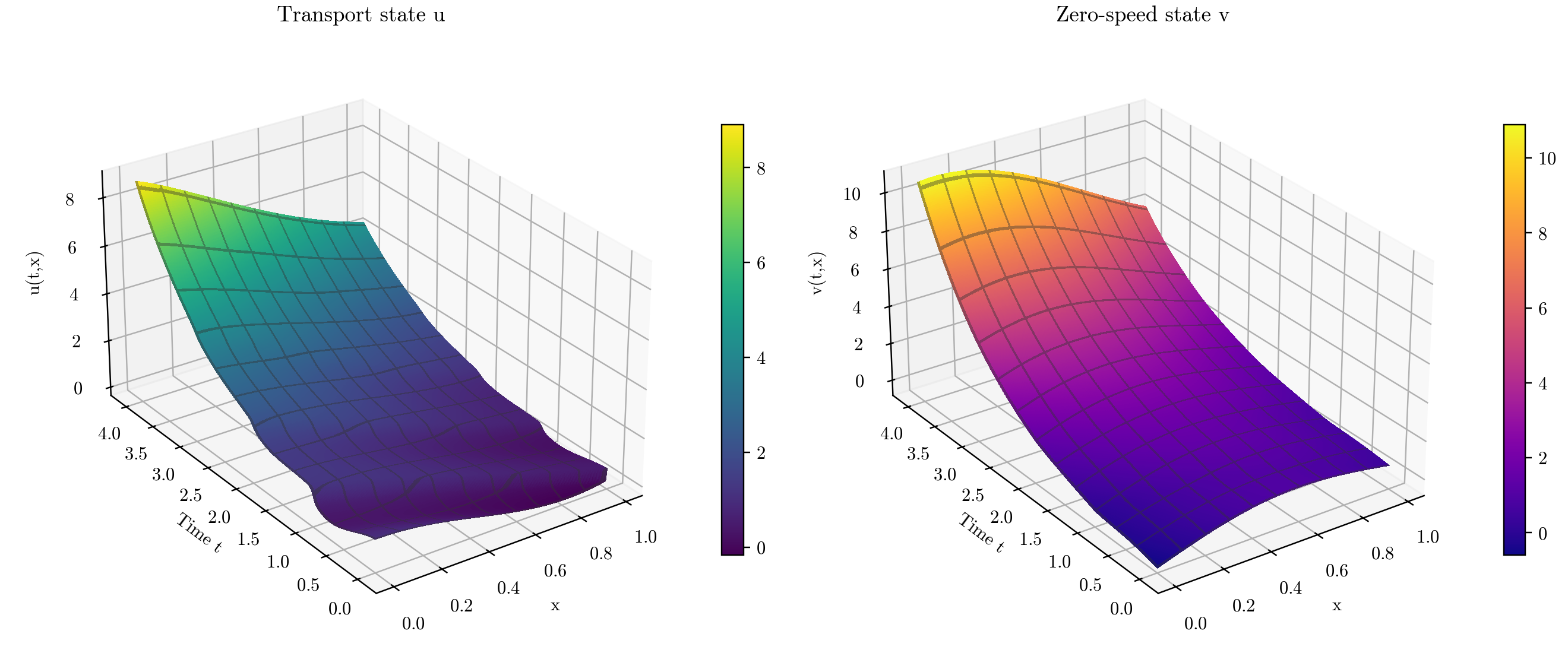}
\caption{State evolution under the proportional feedback
$U(t)=0.5u(t,0)$.  Left: transport state $u$.  Right: zero-speed state
$v$.}
\label{fig:proportional-feedback-states}
\end{figure}

\begin{figure}[htbp]
\centering
\begin{minipage}[c]{0.49\textwidth}
\centering
\includegraphics[width=\linewidth]{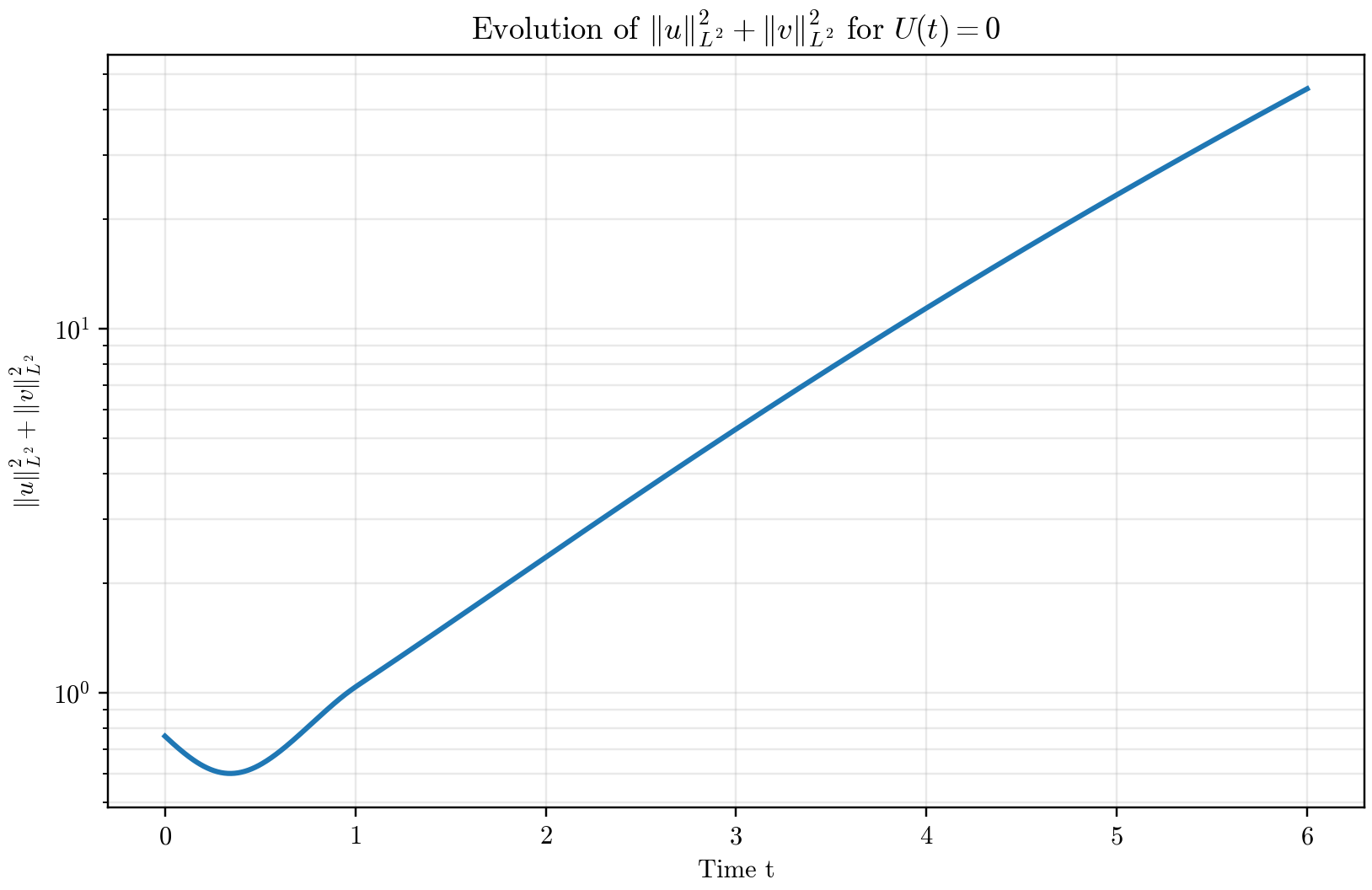}
\end{minipage}
\hfill
\begin{minipage}[c]{0.49\textwidth}
\centering
\includegraphics[width=\linewidth]{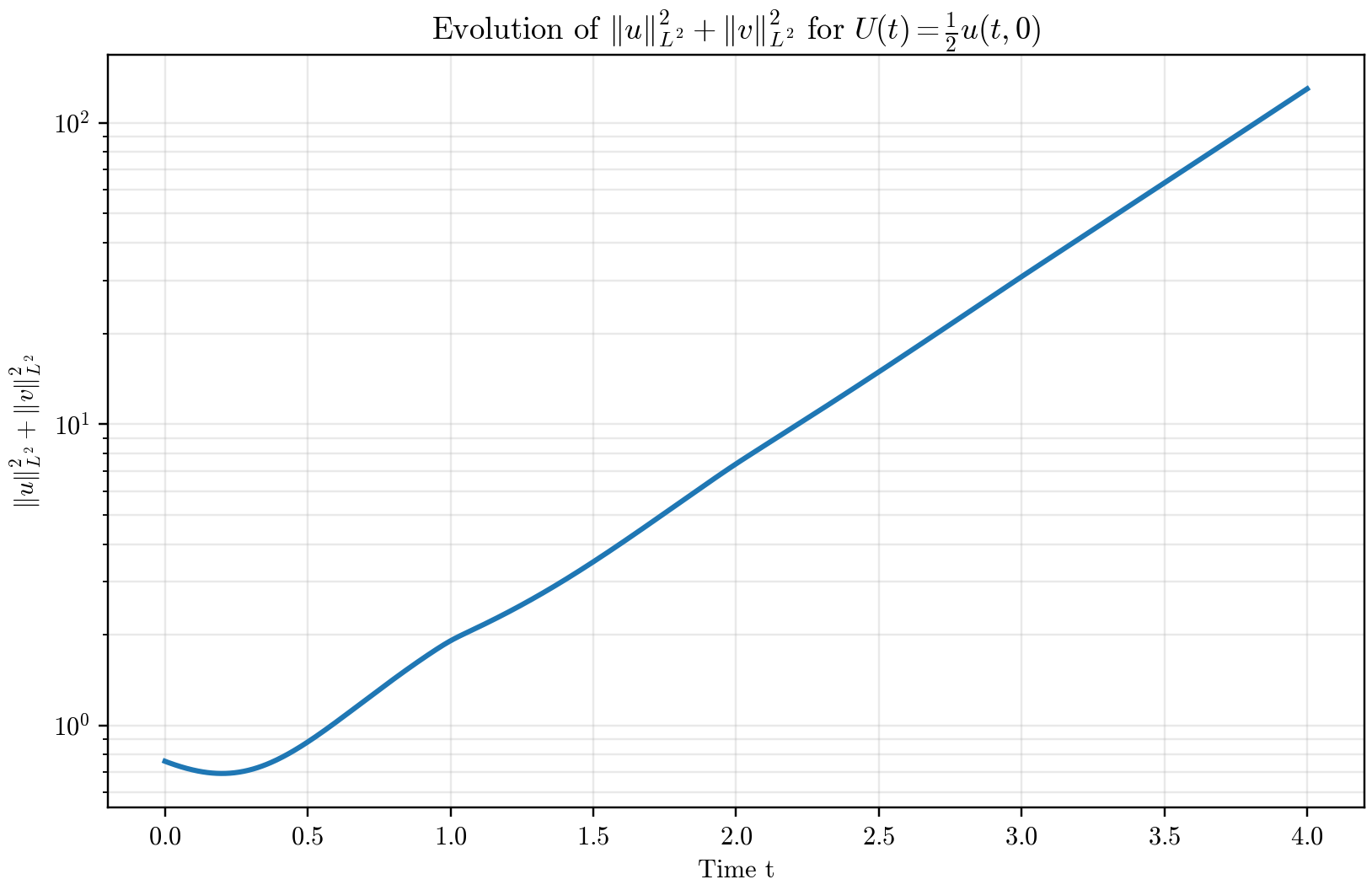}
\end{minipage}
\caption{Comparison of the $L^2$-energies on semilogarithmic scales.
Left: zero boundary input $U(t)=0$.  Right: proportional feedback
$U(t)=0.5u(t,0)$.}
\label{fig:unstable-energy-comparison}
\end{figure}
\FloatBarrier

Finally, we apply the backstepping feedback
\eqref{eq:explicit-feedback-law}.

Figure~\ref{fig:backstepping-W0-states} shows the resulting closed-loop
state evolution.  Both components decrease in amplitude.
Figure~\ref{fig:backstepping-W0-energy-control} shows the corresponding
$L^2$-energy and boundary control over longer time intervals.  The
sustained energy decay is consistent with the asymptotic stability
established in Theorem~\ref{thm:main}.

\begin{figure}[htbp]
\centering
\includegraphics[width=0.98\textwidth]{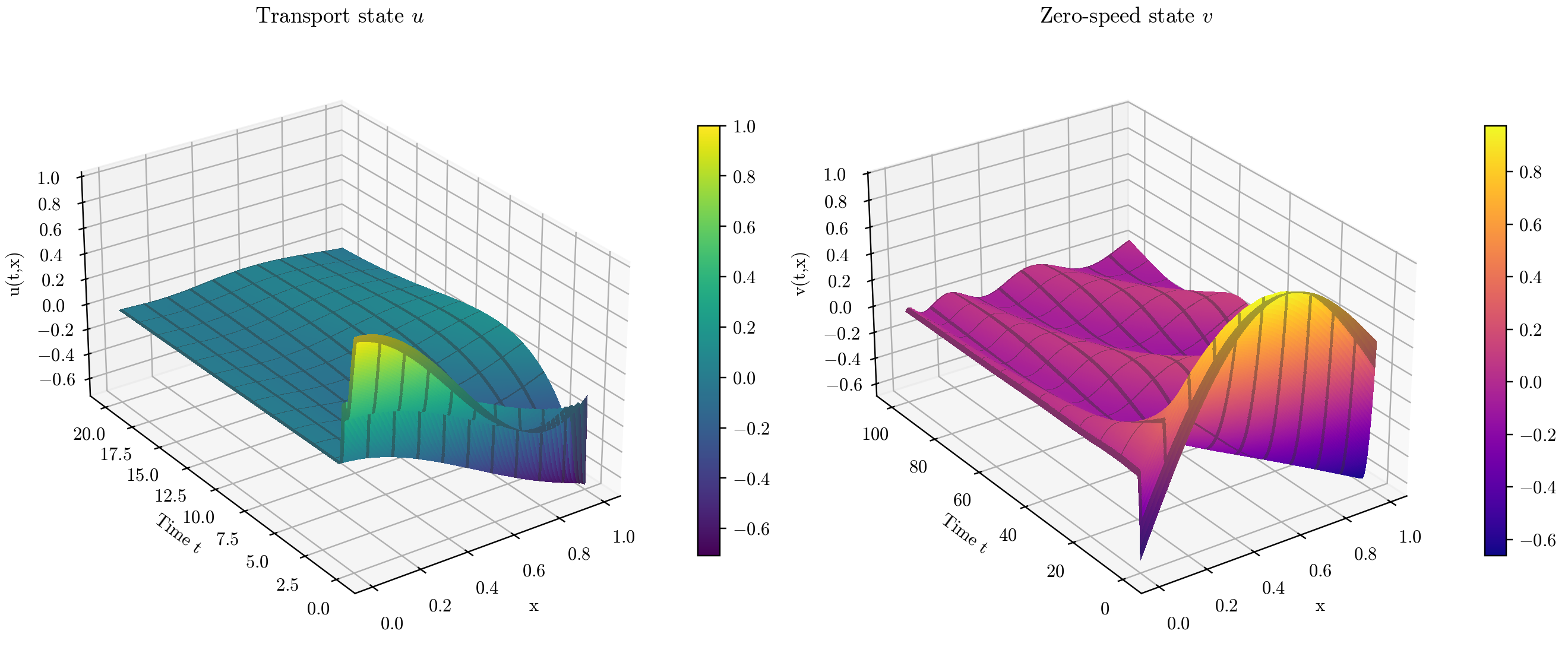}
\caption{Closed-loop state evolution under the backstepping feedback
\eqref{eq:explicit-feedback-law}.
Left: transport state $u$.  Right: zero-speed state $v$.}
\label{fig:backstepping-W0-states}
\end{figure}
\FloatBarrier

\begin{figure}[htbp]
\centering
\begin{minipage}[c]{0.49\textwidth}
\centering
\includegraphics[width=\linewidth]{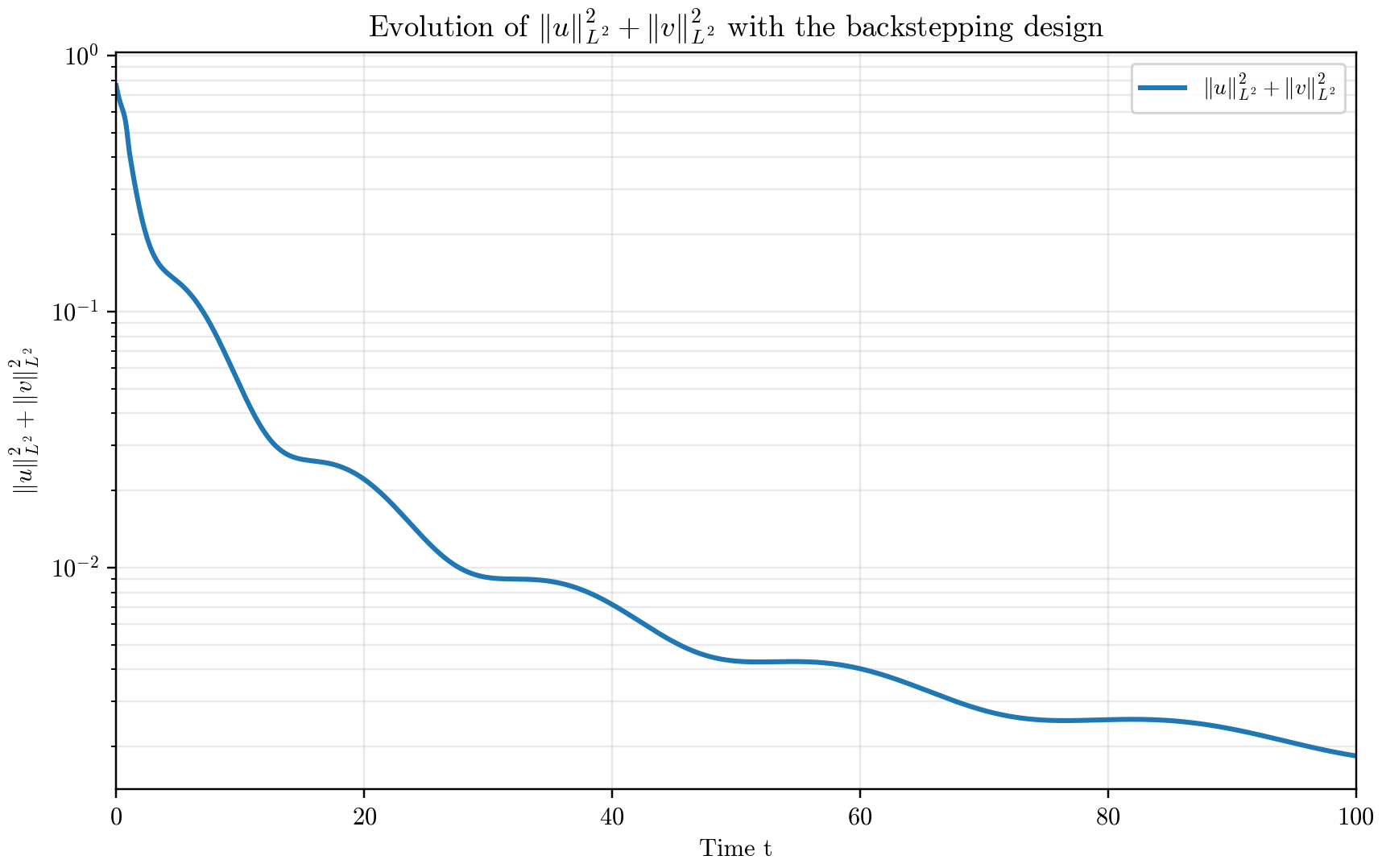}
\end{minipage}
\hfill
\begin{minipage}[c]{0.49\textwidth}
\centering
\includegraphics[width=\linewidth]{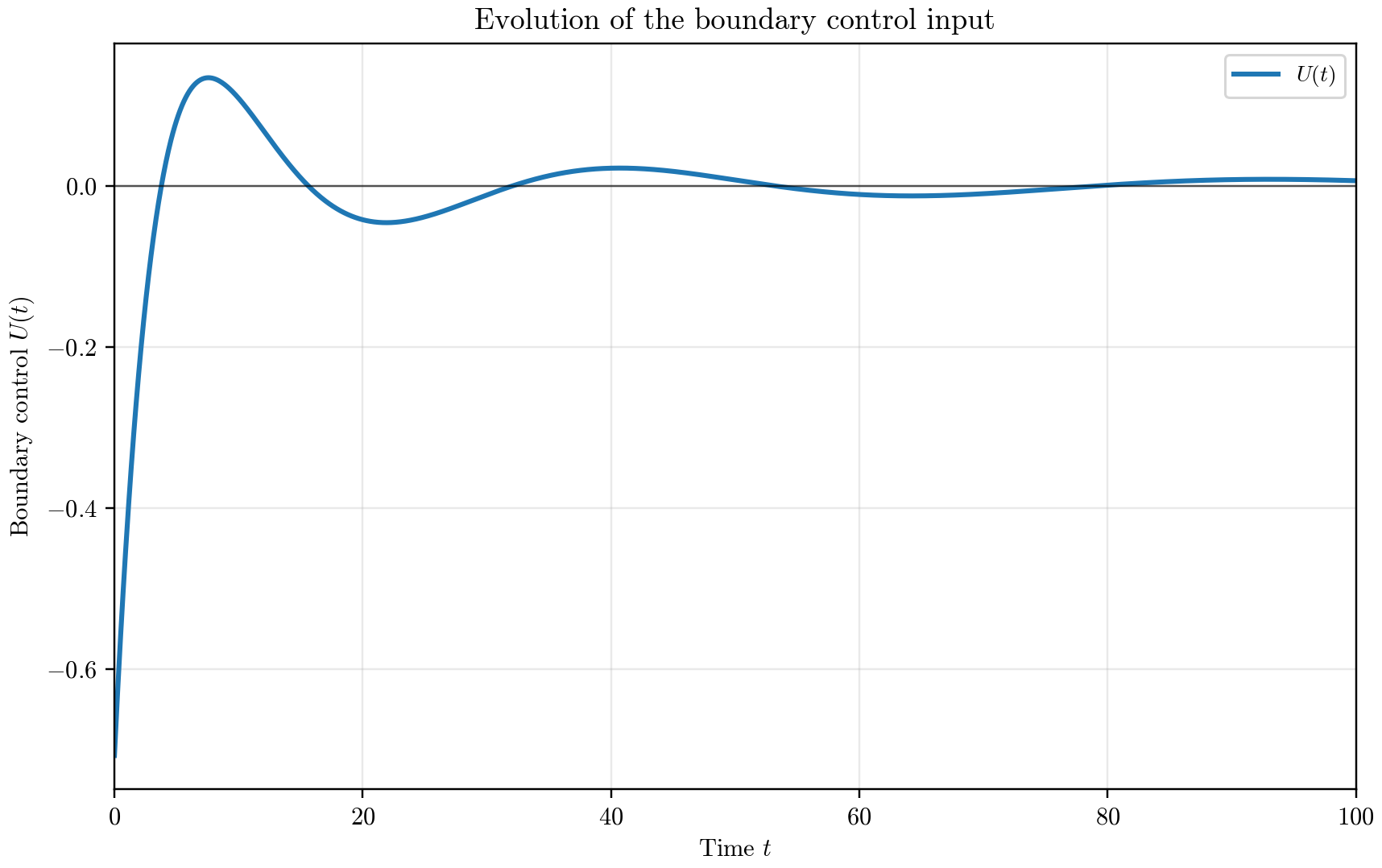}
\end{minipage}
\caption{Closed-loop behavior under the backstepping feedback with
homogeneous target boundary input.  Left: $L^2$-energy of the original
states.  Right: boundary control $U(t)$.}
\label{fig:backstepping-W0-energy-control}
\end{figure}
\FloatBarrier


\bibliography{reference}
\bibliographystyle{alpha}

\end{document}